\documentclass[11pt]{article}
\usepackage{amsmath,amssymb,amsthm}
\usepackage{mathrsfs}
\usepackage{enumerate}
\usepackage{hyperref}
\usepackage{xcolor}
\newtheorem{theorem}{\bf Theorem}[section]
\newtheorem{lemma}[theorem]{\bf Lemma}

\newtheorem{corollary}[theorem]{\bf Corollary}
\theoremstyle{definition}

\newtheorem{remark}[theorem]{\bf Remark}

\begin{document}
\title{Quantitative Stability of Two Weakly Interacting Kinks for the Stationary $\phi^6$ Model}
\author{Xin Liao\\xin\_liao@whu.edu.cn}
\maketitle
\begin{abstract}
Let $H_{0,1}(x)$ and $H_{-1,0}(x)$ denote the kink solutions of the stationary $\phi^6$ model
\[
-\phi''(x)+2\phi(x)-8\phi^3(x)+6\phi^5(x)=0, \quad x\in \mathbb{R}.
\]
In this paper, we establish  quantitative stability estimates for configurations close to two weakly interacting kinks. More precisely, we show that there exist constants $a>0$ and $\varepsilon>0$ such that for any function $u\in L^{\infty}$, if $x_2-x_1>a$ and
\[
\|u - H_{0,1}(\cdot+x_1) - H_{-1,0}(\cdot+x_2)\|_{H^1} < \varepsilon,
\]
then the following estimate holds:
\[
\inf_{y_1,y_2} \| u - H_{0,1}(\cdot+y_1) - H_{-1,0}(\cdot+y_2) \|_{H^1} \lesssim \| \mathscr{F}(u)\|_{L^2},
\]
where 
\[
\mathscr{F}(u) := u'' - 2u + 8u^3 - 6u^5.
\]

\end{abstract}

\section{Introduction}
\subsection{Background}
The $\phi^6$ model equation is a nonlinear wave equation given by
  \begin{equation}\label{phi6}
\partial_{tt}\phi(t,x)-\partial_{xx}\phi(t,x)+2\phi(t,x)-8\phi^3(t,x)+6\phi^5(t,x)=0
  \end{equation}
  for $(t,x)\in \mathbb{R}\times \mathbb{R}$.
  The associated potential energy
 $E_{pot}$ ,  kinetic energy $E_{kin}$ and total energy $E_{tot}$   are defined as
\begin{align}
&E_{pot}(\phi)=\frac{1}{2}\int_{\mathbb{R}} \partial_{x}\phi(t,x)^2 dx +\int_{\mathbb{R}}\phi(t,x)^2(1-\phi(t,x)^2)^2 dx,\\
&E_{kin}(\phi)=\frac{1}{2}\int_{\mathbb{R}} \partial_{t}\phi(t,x)^2 dx,\\
&E_{tot}(\phi,\partial_{t}\phi)=\frac{1}{2}\int_{\mathbb{R}} [\partial_{x}\phi(t,x)^2+ \partial_{t}\phi(t,x)^2] dx +\int_{\mathbb{R}}\phi(t,x)^2(1-\phi(t,x)^2)^2 dx.
\end{align}
The  constant solutions of \eqref{phi6} in the energy space are $\phi = \pm 1$ and $\phi = 0$. 
It is known that the only non-constant stationary solutions with finite total energy are topological solitons, called kinks and anti-kinks; for more details, see Chapter 5 of \cite{Manton2004}.  
Up to translations, the kinks are given by
\begin{equation}
H_{0,1}(x) = \frac{e^{\sqrt{2}x}}{\bigl(1+e^{2\sqrt{2}x}\bigr)^{1/2}}, \quad
H_{-1,0}(x) = -\frac{e^{-\sqrt{2}x}}{\bigl(1+e^{-2\sqrt{2}x}\bigr)^{1/2}},
\end{equation}
and the anti-kinks are $-H_{0,1}(x)$ and $H_{0,1}(-x)$. 

 Define the set 
\begin{equation}
S:=\{u\in L^{\infty}\mid u-H_{0,1}(x)-H_{-1,0}(x)\in H^1 \}.
\end{equation}
In \cite{Moutinho}, A. Moutinho  have proved that 
\begin{theorem}\label{thmA}
There exist $\delta_0>0$, such that if $\varepsilon<\delta_0$, and 
\[ (\phi(0), \partial_{t}\phi(0))\in S\times L^2 \]
with $E_{t}(\phi(0),\partial_{t}\phi(0))=2E_{p}(H_{0,1})+\varepsilon$, then there exist functions $y_1,y_2\in C^2(\mathbb{R})$ such that for all $t\in \mathbb{R}$, the unique global time solution $\phi(t,x)$ of  equation \eqref{phi6} is given by 
\begin{equation}
\phi(t)=H_{0,1}(x+x_1(t))+H_{-1,0}(x+x_2(t))+g(t)
\end{equation}
satisfying
\begin{equation}
e^{-\sqrt{2}(x_2(t)-x_1(t))}+\|g(t)\|_{H^1}^2\lesssim \varepsilon,\quad \forall t\in\mathbb{R}.
\end{equation}
\end{theorem}

\subsection{Main result}
In this paper, we focus on the stationary equation associated with \eqref{phi6}:
\begin{equation}\label{stationary}
-\phi''(x) + 2\phi(x) - 8\phi^3(x) + 6\phi^5(x) = 0, \quad x \in \mathbb{R}.
\end{equation}

We study functions that are close to two weakly interacting kinks $H_{0,1}(x+x_1)$ and $H_{-1,0}(x+x_2)$ with $x_2 - x_1 > 0$ sufficiently large. More precisely, we prove the following quantitative stability result:
\begin{theorem}\label{main res}
There exists  $a>0$ and $\varepsilon>0$, such that for  any function $u\in L^{\infty}$, if there exists  $x_2-x_1>a$,    with   $$\|u-H_{0,1}(x+x_1)-H_{-1,0}(x+x_2)\|_{H^1}<\varepsilon,$$ then  there exist constant $y_{1}, y_{2}$, such that
  \begin{equation*}
   \| u(x)-H_{0,1}(x+y_1)-H_{-1,0}(x+y_2) \|_{H^1}+e^{-\sqrt{2}(y_2-y_1)}\lesssim \| \mathscr{F}(u)\|_{L^2},
  \end{equation*} 
where $ \mathscr{F}(u):= u''(x)-2u(x)+8u^3(x)-6u^5(x)$.
\end{theorem}

As a consequence of Theorem \ref{thmA} and \ref{main res},  we obtain the following corollary for the equation \eqref{phi6}:
\begin{corollary}
There exist $\delta_0>0$, such that if $\varepsilon<\delta_0$, and 
\[ (\phi(0), \partial_{t}\phi(0))\in S\times L^2 \]
with $E_{t}(\phi(0),\partial_{t}\phi(0))=2E_{p}(H_{0,1})+\varepsilon$, then there exist functions $y_1,y_2\in C^2(\mathbb{R})$ such that, for all $t\in \mathbb{R}$, the unique global time solution $\phi(t,x)$ of  equation \eqref{phi6} is given by 
\begin{equation}
\phi(t)=H_{0,1}(x+x_1(t))+H_{-1,0}(x+x_2(t))+g(t)
\end{equation}
satisfying
\begin{equation}
e^{-\sqrt{2}(x_2(t)-x_1(t))}+\|g(t)\|_{H^1}\lesssim \|\partial_{tt}\phi\|_{L_{x}^2},\quad \forall t\in\mathbb{R}.
\end{equation} 
\end{corollary}

\section{Preliminary Lemmas}
Prior to giving the detailed proof of Theorem \ref{main res}, we introduce several fundamental lemmas.
\begin{lemma}\label{pre}
Let $a>1$ and $\alpha\geq \beta>0$. Then we have 
\begin{equation}
-\int_{\mathbb{R}} H_{0,1}^{\alpha}(x)H_{-1,0}^{\beta}(x+a) dx \approx e^{-\sqrt{2}\beta a}.
\end{equation}
\end{lemma}
\begin{proof}
For the lower bound, we observe that
\begin{equation}
\begin{aligned}
-\int_{\mathbb{R}} H_{0,1}^{\alpha}(x)H_{-1,0}^{\beta}(x+a) dx &\gtrsim \int_{0}^{1} H_{-1,0}^{\beta}(x+a) dx\\&\gtrsim
e^{-\sqrt{2}\beta a}.
\end{aligned}
\end{equation}

For the upper bound, we have
\begin{equation}
\begin{aligned}
-\int_{\mathbb{R}} H_{0,1}^{\alpha}(x)H_{-1,0}^{\beta}(x+a) dx &\lesssim  -\int_{\mathbb{R}}H_{0,1}^{\beta}(x)H_{-1,0}^{\beta}(x+a) dx\\&\lesssim
e^{-\sqrt{2}\beta a}\int_{\mathbb{R}} \frac{1}{(1+e^{2\sqrt{2}x})^{\frac{\beta}{2}}(1+e^{-2\sqrt{2}x})^{\frac{\beta}{2}}}\\&
\lesssim e^{-\sqrt{2}\beta a}.
\end{aligned}
\end{equation}
\end{proof}

The following Modulation Lemma can be found in  \cite[Lemma 1]{Moutinho}:
\begin{lemma}\label{mod}
There exists  $a>0$ and $\varepsilon>0$, such that for  any function $u\in L^{\infty}$, if there exists  $x_2-x_1>a$,    with   $$\|u-H_{0,1}(x+x_1)-H_{-1,0}(x+x_2)\|_{H^1}<\varepsilon,$$ then  there exist constant $y_{1}, y_{2}$, such that
\begin{enumerate}
  \item $|y_2-x_2|+|y_1-x_1|\lesssim \varepsilon$;
  \item the function $$g(x)=u(x)-H_{0,1}(x+y_1)-H_{-1,0}(x+y_2)$$ satisfies $\|g\|_{H^{1}}\lesssim \varepsilon$, and the following orthogonality conditions hold:
\begin{equation}\label{ortho condition1}
  \int_{\mathbb{R}} g(x)H_{0,1}'(x+y_1)dx=\int_{\mathbb{R}} g(x)  H_{-1,0}'(x+y_2) dx=0.
\end{equation}
\end{enumerate}
\end{lemma}
\begin{remark} 
 The orthogonality conditions in \eqref{ortho condition1} follow by differentiating the function $$\|u-H_{0,1}(x+y_1)-H_{-1,0}(x+y_2)\|_{L^2}^2$$ with respect to $y_1$ and $y_2$. The modulation parameters $y_1, y_2$ in Lemma \ref{mod} thus realize the best $L^2$ approximation of $u$ by the two-kink manifold. One may also consider the best $H^1$ approximation, which leads to slightly different orthogonality conditions.
\end{remark}

We note that the functions $$H_{0,1}'(x)=\frac{\sqrt{2}e^{\sqrt{2}x}}{(1+e^{2\sqrt{2}x})^{\frac{3}{2}}},\quad H_{-1,0}'(x)=\frac{\sqrt{2}e^{-\sqrt{2}x}}{(1+e^{-2\sqrt{2}x})^{\frac{3}{2}}} $$
are strictly positive and decay exponentially.  Therefore,  by the Perron–Frobenius
theorem, up to a multiplicative constant,  $H_{0,1},$ ($H_{-1,0}'$, respectively)  is the unique $H^1$ solution to the 
equation 
\begin{equation}
-u''+2u-24H_{0,1}^2u+30H_{0,1}^4u=0.
\end{equation} 
$$(-u''+2u-24H_{-1,0}^2u+30H_{-1,0}^4u=0, \mbox{ respectively}.)$$

Moreover,  we have the following lemma:
\begin{lemma}\label{non}
There exists a constant $a>0$ such that for any $y>a$, if we define $$V_{y}(x)=-24(H_{0,1}(x)+H_{-1,0}(x+y))^2+30(H_{0,1}(x)+H_{-1,0}(x+y))^4$$
 and    $g$ is a solution to 
$$-g'' +2g+V_{y}g =\varphi$$
satisfying the orthogonality conditions
\begin{equation}\label{ortho condition}
  \int_{\mathbb{R}} g(x) H_{0,1}'(x)dx=\int_{\mathbb{R}} g(x) H_{-1,0}'(x+y)dx=0,
\end{equation}
then we have the estimate
\begin{equation}
\|g\|_{H^1}\lesssim \|\varphi\|_{L^2}.
\end{equation}
\end{lemma}
\begin{proof}
  Suppose on the contrary, there exist $g_{n},\varphi_n, y_n $ with $$\|g_n\|_{H^1}=1,~ \|\varphi_n\|_{L^2}\to 0~ \mbox{ and } y_n\to+\infty,$$ such that 
  \begin{equation}\label{777}
  -g_{n}''+2g_n+V_{y_n}g_n=\varphi_n
  \end{equation}
  with the orthogonality conditions
  $$\int_{\mathbb{R}} g_n H_{0,1}'(x)dx=\int_{\mathbb{R}} g_n(x) H_{-1,0}'(x+y_n)dx=0.$$
  
  Up to a subsequence,  we may assume that  $g_n\rightharpoonup g_0$  weakly in $H^{1}$.  Then $g_0$ satisfies
   \begin{equation}\label{111}
  -\partial_{xx}g_{0}+2g_{0}-24H_{0,1}^2g_{0}+30H_{0,1}^4g_{0}=0, \quad   \int_{\mathbb{R}} g(x) H_{0,1}'(x)dx=0.
  \end{equation}
  By the Perron–Frobenius property, it follows that $g_0 = 0$. Consequently, 
  \[
g_n \to 0 \quad \text{in } L^2(B_{1000}),
\]
and by a similar argument,
\[
g_n \to 0 \quad \text{in } L^2(B_{1000}(-y_n)).
\] 
Testing \eqref{777} with $g_n$, we obtain
  \begin{equation}
  \begin{aligned}
o(1)&=  1+\int_{\mathbb{R}}  (1+V_{y_n})     |g_{n}|^2 dx \\&
=1+o(1)+ \Big{(}\int_{-\infty}^{-y_n-1000} + \int_{1000}^{+\infty}+ \int_{-y_n+1000}^{-1000} \Big{)}   (1+V_{y_n})     |g_{n}|^2 dx .
  \end{aligned}
  \end{equation}
  In the region $(-\infty, -y_n-1000)\cup (1000, +\infty)$, we have $$ V_{y_n}>0,$$
 and  in the region $(-y_n+1000, -1000)$, we have  $$ |V_{y_n}|<\frac{1}{1000}.$$
  therefore, we deduce that
  $$o(1)>1,$$
  which is a contradiction.
\end{proof}

\section{Proof of Theorem \ref{main res}}
\begin{proof}[Proof of Theorem \ref{main res}]
Denote $$H_1=H_{0,1}(x+y_1),\quad H_2=H_{-1,0}(x+y_2),$$
where $y_1,y_2$ are the modulation parameters obtained from Lemma  \ref{mod}, and  set 
$$g:=u-H_1-H_2.$$
Then $g$ satisfies
\begin{equation}\label{main}
-g''+2g-24(H_1+H_2)^2g+30(H_1+H_2)^4g=\mathscr{F}(u)+h+N(g),
\end{equation}
where the interaction term is $$h:=8(H_1+H_2)^3-8H_1^3-8H_2^3+6H_1^5+6H_{2}^5-6(H_1+H_2)^5,$$
and the nonlinear term \begin{align*}N(g):&=24(H_1+H_2)g^2-60(H_1+H_2)^3g^2\\&+8g^3-60(H_1+H_2)^2g^3-30(H_1+H_2)g^4-6g^5.\end{align*}

Testing equation \eqref{main} with $H_{1}'$ and integrating by parts, we obtain
\begin{equation}\label{111111111111}
\begin{aligned}
 \Big{|}\int_{\mathbb{R}} h H_{1}' dx\Big{|}&\lesssim \|\mathscr{F}(u)\|_{L^2}+\|g\|_{H^1}^2\\&+ \Big{|}\int_{\mathbb{R}} [30(H_1+H_2)^4-30H_1^4+ 24(H_1+H_2)^2-24H_1^2 ]g H_{1}'  dx\Big{|}
\end{aligned}
\end{equation}
Since $H_1'\lesssim H_1$, by Lemma \ref{pre}, we obtain that
\begin{equation}
\begin{aligned}
\int_{\mathbb{R}} h H_{1}' dx=24 \int H_1^2 H_{1}' H_2 dx-30 \int  H_1^4 H_{1}' H_2 dx+o(e^{-\sqrt{2}(y_2-y_1)}).
\end{aligned}
\end{equation}
A direct computation shows that
\begin{equation}
\begin{aligned}
 \int H_1^2 H_{1}' H_2 dx&=-e^{-\sqrt{2}(y_2-y_1)}\int_{-\infty}^{+\infty} 
\frac{\sqrt{2} e^{2\sqrt{2}\, x } }{
(1+e^{2\sqrt{2} x})^{2.5}\,
(1+e^{-2\sqrt{2} x-(y_2-y_1)})^{0.5}
}\,dx
\\&=-e^{-\sqrt{2}(y_2-y_1)}(\int_{-\infty}^{+\infty} \frac{\sqrt{2} e^{2\sqrt{2}\, x } }{
(1+e^{2\sqrt{2} x})^{2.5}\,
}\,dx +o(1))
\\&=
-e^{-\sqrt{2}(y_2-y_1)}(\int_{1}^{+\infty}\frac{1}{2t^{2.5}} dt+o(1)  )
\\&=
-e^{-\sqrt{2}(y_2-y_1)}(\frac{1}{3}+o(1)).
\end{aligned}
\end{equation}
While 
\begin{equation}
\begin{aligned}
\int  H_1^4 H_{1}' H_2 dx&=-e^{-\sqrt{2}(y_2-y_1)}\int_{-\infty}^{+\infty} 
\frac{\sqrt{2} e^{4\sqrt{2}\, x } }{
(1+e^{2\sqrt{2} x})^{3.5}\,
(1+e^{-2\sqrt{2} x-(y_2-y_1)})^{0.5}
}\,dx\\&
=-e^{-\sqrt{2}(y_2-y_1)}(\int_{0}^{\infty}\frac{t}{2(1+t)^{3.5}} dt+o(1))\\&
=-e^{-\sqrt{2}(y_2-y_1)}(\frac{2}{15}+o(1)).
\end{aligned}
\end{equation}
Therefore,  
\begin{equation}
 \Big{|}\int_{\mathbb{R}} h H_{1}' dx\Big{|} \approx e^{-\sqrt{2}(y_2-y_1)}.
\end{equation}
For the third term on the right-hand side of equation \eqref{111111111111}, we apply H\"older's inequality and use the smallness of $\|g\|_{H^1}$ to obtain
\begin{equation}
\begin{aligned}
\Big{|}\int_{\mathbb{R}} [30(H_1+H_2)^4-30H_1^4+ &24(H_1+H_2)^2-24H_1^2 ]g H_{1}'  dx\Big{|}
\lesssim 
\int_{\mathbb{R}} |H_{1}H_{2}g| dx\\&
\lesssim \|g\|_{H^1}e^{-\sqrt{2}(y_2-y_1)}=o(e^{-\sqrt{2}(y_2-y_1)}).
\end{aligned}
\end{equation}
Thus, we derive from  equation \eqref{111111111111} that
\begin{equation}
e^{-\sqrt{2}(y_2-y_1)}\lesssim  \|\mathscr{F}(u)\|_{L^2}+o(\|g\|_{H^1}).
\end{equation}
On the other hand, 
Lemma \ref{non} shows that
\begin{equation}
\|g\|_{H^1}\lesssim \|\mathscr{F}(u)+h+N(g)  \|_{L^2}\lesssim \|\mathscr{F}(u)\|_{L^2}+e^{-\sqrt{2}(y_2-y_1)}+o(\|g\|_{H^1}).
\end{equation}
We thus obtain that
\begin{equation}
 \|g\|_{H^1}\lesssim \|\mathscr{F}(u)\|_{L^2}, \quad e^{-\sqrt{2}(y_2-y_1)}\lesssim  \|\mathscr{F}(u)\|_{L^2}.
\end{equation}
\end{proof}

\end{document}